\documentclass[11pt]{article}

\usepackage[T1]{fontenc}
\usepackage[latin1]{inputenc}
\usepackage{amsthm}
\usepackage{amsmath}
\usepackage{amssymb}
\usepackage{ifdraft}
\usepackage[x11names]{xcolor}

\usepackage[textsize=footnotesize,
            shadow,
            obeyDraft,
            backgroundcolor=LightSalmon1,
            bordercolor=Sienna3,
            linecolor=Sienna3]{todonotes}
\ifdraft{%
  \setlength{\marginparwidth}{2.3cm}% For todo notes
  \usepackage{showlabels}%
  
}{}

\ifdraft{%
  \definecolor{new}{rgb}{0.9,0.1,0.3}%
  \definecolor{felix}{named}{SteelBlue4}%
}{%
  \definecolor{new}{named}{black}%
}

\usetikzlibrary{matrix}
\usepackage{enumerate}

\newtheorem{thm}{Theorem}
\newtheorem{defn}{Definition}
\newtheorem{prop}{Proposition}

\newtheorem{rem}{Remark}
\newtheorem{ex}{Example}
\newtheorem{lem}{Lemma}
\def\proof{\par\noindent{\bf Proof.\ } \ignorespaces}

\newtheorem{alg}{Algorithm}

\newcommand{\MM}{{\mathcal M}}
\newcommand{\NN}{{\mathbb N}}
\newcommand{\RR}{{\mathbb R}}
\newcommand{\ZZ}{{\mathbb Z}}

\newcommand{\diag}[1] {\mathsf{diag}( #1 )}

\newcommand{\Sig}{\mathcal{S}}
\newcommand{\Beh}{\mathcal{B}}
\newcommand{\Op}{\mathcal{O}}
\newcommand{\OpFr}{\overline{\Op}}

\DeclareMathOperator{\im}{im}
\DeclareMathOperator{\id}{id}
\newcommand{\pdeg}{\mathop{\operatorname{deg}_\partial}}
\renewcommand{\(}{(\!(}
\renewcommand{\)}{)\!)}
\newcommand{\ID}[1][\relax]{\mathbf{1}\ifx#1\relax\relax\else_{#1}\fi}
\newcommand{\ZERO}[1][\relax]{%
  \mathbf{0}\ifx#1\relax\relax\else_{\mytimes#1}\fi}
\gdef\mytimes#1,#2{#1\times#2}
\newcommand{\Mat}[3]{{#1}^{#2\times#3}}             % Matrices
\newcommand{\MatGr}[2]{\operatorname{Gl}_{#2}(#1)}  % Unimodular matrices
\newcommand{\RV}[2]{{#1}^{1\times#2}}               % Row vectors
\newcommand{\CV}[2]{{#1}^{#2\times1}}               % Column vectors
\newcommand{\RS}[3][R]{\RV{#1}{#3}{#2}}             % Row space
\newcommand{\Mod}[4][R]{\mathchoice%                  Quotient space
  {\frac{\RV{#1}{#4}}{\RS[#1]{#2}{#3}}}%
  {{\RV{#1}{#4}}/{\RS[#1]{#2}{#3}}}%
  {{\RV{#1}{#4}}/{\RS[#1]{#2}{#3}}}%
  {{\RV{#1}{#4}}/{\RS[#1]{#2}{#3}}}}
\newcommand{\row}[2]{{#2}_{#1,*}}
\newcommand{\col}[2]{{#2}_{*,#1}}
\newcommand{\eg}{e.\,g.}
\newcommand{\ie}{i.\,e.}
\newcommand{\wrt}{w.\,r.\,t.}
\newcommand{\qtext}[1]{\quad\text{#1}\quad}
\newcommand{\qqtext}[1]{\qquad\text{#1}\qquad}
\newcommand{\WLOG}{w.\,l.\,o.\,g.}
\renewcommand{\th}{-th}
\newcommand{\GB}{Gr{\"o}bner basis}

\newcommand{\Rcal}{\mathcal{R}}

\title{On the computation of $\pi$-flat outputs for linear time-varying differential-delay systems} 

\author{Felix Antritter\thanks{Automatisierungs- und Regelungstechnik, EIT8.1, Universit\"at der Bundeswehr M\"unchen, Germany. E-mail: \texttt{felix.antritter@unibw.de}}
\and
Franck Cazaurang\thanks{Univ. Bordeaux, IMS, UMR 5218, 33405 Talence cedex, France. E-mail: \texttt{franck.cazaurang@u-bordeaux1.fr}}
\and
Jean L{\'e}vine\thanks{CAS-Math{\'e}matiques et Syst{\`e}mes, MINES-ParisTech, 35 rue Saint-Honor{\'e}, 77300 Fontainebleau, France. E-mail: \texttt{jean.levine@mines-paristech.fr}}
\and
Johannes Middeke\thanks{Upper Austrian government. Research Institute for Symbolic Computation (RISC), Johannes Kepler University, Altenbergerstra{\ss}e~52, A-4040~Linz, Austria. E-mail: \texttt{jmiddeke@risc.jku.at}}
~\footnote{J.~Middeke was supported by the Austrian Science Foundation (FWF) under the project DIFFOP (P\,20\,336-N18) and by the strategic programme ``Innovatives O{\"O} 2010plus'' by the Upper Austrian government.}
}
\date{}

\begin{document}
\maketitle

\begin{abstract}
We introduce a new definition of $\pi$-flatness for linear differential delay systems with time-varying coefficients. We characterize $\pi$- and $\pi$-0-flat outputs and provide an algorithm to efficiently compute such outputs. We present an academic example of motion planning to discuss the pertinence of the approach.
\end{abstract}


%\end{frontmatter}

\section{Introduction}
\color{new}In this paper, we consider systems of the kind
\begin{equation}
A x=B u,\label{sys_lin_2}
\end{equation}
with the  state, or partial state, $x$ of dimension $n$, input $u$ of dimension $m$ and matrices $A\in\Mat{\Op}nn$ and $B\in\Mat{\Op}nm$, where $\Op$ is the polynomial ring of operators $\delta$ and $\partial$, where:
\begin{itemize} 
\item $\delta$ is the shift, or time-delay, operator, defined by $\delta f(t)= f(t-\tau)$, with $\tau \in \RR_{+}$  given, for all continuous functions $f$ and all $t$, 
\item $\partial \triangleq \frac{d}{dt}$ is the ordinary time derivative, 
\item the coefficients belong to the field $K$ of real meromorphic functions of time. 
\end{itemize}
The ring $\Op$ will be denoted by $K[\delta, \partial]$ in the sequel (see Section~\ref{op-sig:sec} for a detailed presentation).

A typical example of such systems is given by:
\begin{ex}\label{intro:ex}
\begin{equation}\label{ex:eq}
\begin{aligned}
  \dot{x}_1(t) &= a(t)\bigl(x_2(t-1) - x_2(t-2)\bigr) \\
  \dot{x}_2(t) &= u(t-1),
\end{aligned}
\end{equation}
where the state $x\triangleq (x_{1},x_{2})$ is 2 dimensional, the control $u$ is scalar, and the function $a$ is \emph{meromorphic}\footnote{rational fraction of analytic functions} with respect to time. Here, the matrices $A$ and $B$ are given by
$$A \triangleq \left( \begin{array}{cc}\partial &-a(t)(\delta-\delta^{2})\\0&\partial
\end{array}\right), \quad
B \triangleq \left( \begin{array}{c}0\\\delta
\end{array}\right).
$$
\end{ex}
Our aim in this paper is to extend existing results on the notion of flatness, to such systems, namely the class of linear time-varying differential delay systems with meromorphic time dependence, and to use this property for motion planning.

Let us briefly sketch a historical overview of the concepts of differential flatness, freeness, $\pi$-freeness and $\pi$-flatness.

For ordinary differential systems (without delays) \emph{differential flatness}, roughly speaking, means that all the variables of an under-determined set of ordinary differential equations can be expressed as functions of a particular output, called \emph{flat output}, and a finite number of its successive time derivatives (\cite{Martin_92,Fliess_95,Fliess_99}, see also  \cite{Ramirez_04,Levine_09,Levine_11} and the references therein). In the particular case of \emph{linear} differential systems, relations with the 
behavioral approach of \cite{Polderman_98}, or more precisely with \emph{latent variables} of \emph{observable image representations} \cite{Trentelman_04} have been established (see also \cite{Zerz_06}). In this case, differential flatness is equivalent to the \emph{freeness} of the so-called \emph{system module} \cite{Fliess_90,Fliess_92}, which characterizes the system \emph{controllability}.

For \emph{linear time-invariant differential delay systems} and more general classes of infinite-dimensional systems, only the concepts of \emph{freeness}, with several variants such as torsion freeness, projective freeness and $\pi$-freeness, have been proposed  and thoroughly discussed in \cite{Mounier_95,Fliess_96,Rudolph_03,Chyzak_05,Chyzak_07}, without  explicit reference to the notion of flatness. To the authors knowledge, the first explicit appearance, in this context, of the notion of $\pi$-flatness, and in fact, to be more precise, of $\delta$-flatness, is in \cite{Petit_97,Petit_00}. 
Then, circumventing serious algebraic obstructions recalled in Section~\ref{sys:sec}, the notion of freeness has been extended by \cite{Chyzak_05,Chyzak_07, Robertz_14} to the context of linear time-varying differential-delay systems with polynomial time dependence. For linear infinite-dimensional systems, relations with the notion of system parameterization are discussed in \cite{Pommaret_99,Pommaret_01} and, with the behavioral approach, in  \cite{Rocha_97}. Interesting control applications of linear differential-delay systems may be found in \cite{Mounier_95,Petit_00,Rudolph_03}.

Characterizing differential flatness and flat outputs has been an active topic since the beginning of this theory. The interested reader may find a historical perspective of this question in \cite{Levine_09,Levine_11}. Constructive algorithms, relying on standard computer algebra environments, may be found e.g. in \cite{Antritter_08,Antritter_10} for nonlinear finite-dimensional systems, or \cite{Antritter_11,Chyzak_04} for linear systems over Ore algebras.

The results and algorithms presented in this paper which, as announced, concern time-varying differential-delay systems with meromorphic time-de\-pendence, may be seen as an extension of those of \cite{Petit_97,Levine_03,Levine_09} to this context. A thorough discussion about the existing system theoretic algebraic approaches and the difficulties to extend them to the present context is provided at the beginning of section~\ref{sys:sec}. Unfortunately, this discussion cannot take place in this introduction because of the complexity of its technical content.

Our main contributions are (1) a new definition of $\pi$- and $\pi$-$k$-flatness, for every $k\in \NN$, applicable to linear time-invariant as well as time-varying systems and suitable for motion planning, with or without distinction between  state and input variables, (2) the characterization of $\pi$- and $\pi$-0-flatness in terms of the \emph{hyper-regularity} of the system matrices, (3) the connection between \emph{hyper-regular} matrices and \emph{row-/column-reduction}, a computationally most useful property, and (4) yielding an elementary algorithm, of polynomial time complexity, to compute $\pi$- and $\pi$-0-flat outputs, based on the row/column reduction of the former matrices. 

Let us emphasize that, with our definition of $\pi$-flatness, all the system trajectories are obtained as functions of a $\pi$-flat output, a finite number of its derivatives, time delays, and a number of predictions defined by the degree of the polynomial~$\pi$, to be compared with the results of \cite{Chyzak_07, Robertz_14}, relative to the aforementioned more restrictive freeness concept for linear time-varying systems with polynomial dependence, where the trajectory parameterization does not involve predictions.

Note that, to describe our class of systems, the ring, denoted by $\Op$, of polynomials in both the delay and  differential operators, which is non principal, constitutes a natural algebraic framework. However, 
the evaluation of our $\pi$- and $\pi$-0-flatness criteria relies upon computations over a principal ideal ring, denoted by $\OpFr$, which contains $\Op$ and aids in making the computations much simpler. 

In addition, our formalism is applicable to systems where the distinction between the input variables, possibly given by the nature of the actuators, and the other system variables (states or partial states), is required or not, at will.

The paper is organized as follows. A first section briefly presents recalls on operators and signals (Section~\ref{op-sig:sec}), followed by recalls on matrices over the non-commutative differential ring $\Op$ (Section~\ref{mat-hyp-reg:sec}). In the latter section, the notions of hyper-regularity and row and column reduction are introduced and characterized in this non-commutative context. Section~\ref{sys:sec} then deals with differential algebraic notions of systems, $\pi$-flatness and $\pi$-$k$-flatness and contains the main results characterizing $\pi$-flat and $\pi$-0-flat outputs. Algorithms to compute such $\pi$-flat and $\pi$-0-flat outputs are then deduced.
Finally, the proposed methodology is illustrated by an example of motion planning in Section~\ref{ex:sec}.

\normalcolor

\section{Operators}\label{op-sig:sec}

In order to deal with differential delay systems, we need some recalls of polynomial algebra.

\subsection{Recalls on Ore polynomials}\label{sec:Ore}
In order to model linear mixed differential and time-delay
operators, we use the so-called Ore
polynomial approach. Ore polynomials are a class of non-commutative polynomials, named
after {O}ystein Ore who was the first to discuss them
in~\cite{OreEng_33}. 

Let $K$ be a ring and let $\sigma\colon K\to K$ be an automorphism. An
additive map $\vartheta\colon K\to K$ is called a
\emph{$\sigma$-derivation} if for all $a$ and $b \in K$ the
\emph{$\sigma$-Leibniz rule}
\begin{math}
  \vartheta(ab) = \sigma(a)\vartheta(b) + \vartheta(a)b
\end{math}
holds (compare with~\cite[Sect.~7.3]{Cohn_03}). Consider the free $K$-left
module generated by the powers of an indeterminate $x$. We define the
right-multiplication of $x$ by an element of $K$ with the
\emph{commutation rule}
\begin{equation*}
  x a = \sigma(a) x + \vartheta(a) \qquad\text{for all } a \in K.
\end{equation*}
Assuming associativity and distributivity, this rule allows us to
compute arbitrary products. It can be shown (see, \eg,
\cite[Thm.~7.3.1]{Cohn_03}) that this makes the free module into a ring
which we call the ring of \emph{(left) Ore polynomials} in $x$ \wrt\
$\sigma$ and $\vartheta$. In the literature this is usually denoted by
$K[x;\sigma,\vartheta]$.

The \emph{degree} of an Ore polynomial $p$ is defined as the largest
exponent $n$ such that $x^n$ has a non-zero coefficient in $p$. We use
$\deg 0 = -\infty$. If $K$ is a domain, then we have the familiar rule
$\deg (pq) = \deg p + \deg q$ for all $p$ and $q \in
K[x;\sigma,\vartheta]$. If $K$ is a division ring, then it is possible
to divide elements in $K[x;\sigma,\vartheta]$ with remainder in a way
which is very similar to the usual polynomial
division. See~\cite{Bronstein_96} for details. This turns
$K[x;\sigma,\vartheta]$ into a (left and right) principal ideal domain
and thus into a (left or right) Ore ring (see also
\cite[Prop.~5.9]{Cohn_00}). Thus, we can form the field of (left or
right) fractions $K(x;\sigma,\vartheta)$. See for example
\cite{Jezek_96} for an extensive introduction on how to properly define
the various arithmetic operations for such fractions.

\medskip

If $K$ is commutative, $\sigma = \id$ and $\vartheta = 0$ are the
identity and the zero map, respectively, then $K[x;\id,0]$ is just
the usual polynomial ring $K[x]$. Two other important special cases of
Ore polynomials are linear differential and delay operators, which we
discuss in the following.
\begin{ex}
  Let $K$ be the field of meromorphic functions over the real line (see e.g. \cite[p. 42]{Cartan_61})).
  \begin{enumerate}
  \item Assume first that $\vartheta = 0$ and that $\sigma = \delta$ is
    the \emph{time delay operator} which is defined by $\delta f(t) =
    f(t-\tau)$ for all $f \in K$ where $\tau > 0$ is a fixed real
    number. In an abuse of notation, we will denote the ring
    $K[x;\delta,0]$ just by $K[\delta]$, \ie, we identify $\delta$
    with the Ore variable $x$. The commutation rule for $K[\delta]$ is
    \begin{math}
      \delta a(t) = a(t-\tau) \delta
    \end{math}
    for all $a \in K$. We call it the ring of \emph{time delay
      operators}.              
  \item Assume now, that $\sigma = \id$. Let $\vartheta = \partial$ be
    the usual derivation in the sense of calculus. Then, the ring
    $K[x;\id,\partial]$---which we will just write as
    $K[\partial]$---has the commutation rule
    \begin{math}
      \partial a(t) = a(t) \partial + \dot{a}(t)
    \end{math}
    for all $a \in K$. This is the ring of \emph{differential
      operators}.
  \end{enumerate}
\end{ex}

Since the maps $\partial$ and $\delta$ commute, we may extend
$\partial$ to $K[\delta]$ by setting $\partial(\delta) = 0$. Thus, the
ring $\Op = K[\delta,\partial]$---or in more complete notation
$K[x;\delta,0][y;\id,\partial]$---is well-defined and has the
commutation rules
\begin{equation*}
  \delta a(t) = a(t-\tau)\delta,
  \qquad
  \partial a(t) = a(t) \partial + \dot{a}(t)
  \qqtext{and}
  \partial \delta = \delta \partial
\end{equation*}
where $a \in K$. We call $\Op$ the ring of time-delay differential
operators. The commutation of $\partial$ and $\delta$ means that this
ring is an Ore algebra in the sense
of~\cite[Def.~1.2]{ChyzakSalvy_98}. \color{new} Using the formul\ae\ in
\cite[Thm.~13]{Jezek_96}, we may extend the action of $\partial$ to
the fractions of $K[\delta]$, denoted by $K(\delta)$. Thus, also the ring $\OpFr =
K(\delta)[\partial]$ is a well-defined Ore polynomial
ring. Since $K(\delta)$ is a division ring, the ring
$\OpFr$ will be Euclidean---just as explained above. This means that
in $\OpFr$ computational tools such as greatest common divisors and
certain matrix normal forms become available---see
Section~\ref{mat-hyp-reg:sec}---which cannot be calculated within
$\Op$. This fact will become very important for our algorithms in
Subsection~\ref{subsec:pi-flat}

\begin{ex}\label{ex:explicit-op}
  An example of a typical element of $\OpFr = K(\delta)[\partial]$
  would be
  \begin{equation*}
    p = (\delta^2 - \delta)^{-1} \partial^2 
    + a (\delta+b)^{-1}(\delta - c)\partial 
    + \delta^2
  \end{equation*}
  where $a,b,c \in K$ are meromorphic functions. Note, that division
  is only allowed by (sub-) expressions containing only $\delta$ but
  no $\partial$.
\end{ex}
\normalcolor

\section{Matrices \&\ Hyper-Regularity}\label{mat-hyp-reg:sec}

We model systems of linear time-delay differential equations using
matrices of operators. The set of all $n\times m$ matrices with
entries in $\Op$ is denoted by $\Mat{\Op}nm$. Square matrices in
$\Mat{\Op}nn$ which possess a two-sided inverse that is also in
$\Mat{\Op}nn$ are called \emph{unimodular}. The set of all unimodular matrices of $\Mat{\Op}nn$ is denoted by
$\MatGr{\Op}n$. We write $\ID[n]$ for the
$n\times n$ identity matrix and $\ZERO[n,m]$ for the $n\times m$ zero
matrix. In both cases we will omit the indices when they are obvious
from the context.  We use $\RV{\Op}m$ for the set of row vectors of
length $m$ and $\CV{\Op}n$ for the set of column vectors of length
$n$. Given a matrix $M \in \Mat{\Op}nm$ we denote its $\Op$-row space
by $\RS[\Op]{M}n$. For a matrix $M = (M_{i,j}) \in \Mat{\Op}nm$ we
define
\begin{math}
  \pdeg M = \max \{ \pdeg M_{i,j} \mid i=1,\ldots,n, j=1,\ldots,m \}.
\end{math}
This will also be applied to row or column vectors regarding them as
$1\times m$- or $n\times 1$-matrices, respectively. We write
$\row{i}M$ for the $i$\th\ row of $M$ and $\col{j}M$ for the $j$\th\
column of $M$ where $i=1,\ldots,n$ and $j=1,\ldots,m$.

\bigskip

Let $M \in \Mat{\Op}nm$. Since $K(\delta)$ is a division ring, $\OpFr$
is a principal ideal domain. This means we can apply
\cite[Thm.~8.1.1]{Cohn_85} in order to find unimodular matrices $S \in
\MatGr{\OpFr}n$ and $T \in \MatGr{\OpFr}m$ such that $SMT$ is in
\emph{Smith-Jacobson form}. This is a diagonal form with the
additional property that each diagonal element is a total divisor of
the next one.

\begin{defn}[Hyper-regularity \cite{Levine_09}]\label{def:hyperr}
  A matrix $M \in \Mat{\Op}nm$ is called \emph{hyper-regular} if the
  diagonal elements of its Smith-Jacobson form are all $1$, \ie, if
  there are unimodular matrices $S \in \MatGr{\OpFr}n$ and $T \in
  \MatGr{\OpFr}m$ such that
  \begin{equation*}
  \begin{aligned}
    n \geq m &
    \qtext{and}
    SMT = \begin{pmatrix} \ID[m] \\ \ZERO[(n-m),m] \end{pmatrix}\\
     \qtext{or} n < m & \qtext{and} SMT = (\ID[n], \ZERO[n,(m-n)]).
     \end{aligned}
  \end{equation*}
\end{defn}

Usually, the given method for obtaining a Smith-Jacobson form involves
computing a diagonal form by repeatedly taking greatest common (left
or right) divisors as a first step. Since the performance of this
approach is difficult to predict and most likely exponential due to
degree growth, we will give an alternate, and computationally more efficient, characterization of
hyper-regularity below.

\begin{prop}[\cite{Antritter_11}]
\begin{itemize}
\item[(i)] A matrix $M \in \Mat{\Op}nm$, with $n < m$, is  \emph{hyper-regular} if, and only if, it possesses a right-inverse, i.e. $\exists$ $\bar{T}\in \MatGr{\OpFr}m$  such that $M\bar{T}= (\ID[n], \ZERO[n,(m-n)])$.
\item[(ii)] A matrix $M \in \Mat{\Op}nm$, with $n \geq  m$, is  \emph{hyper-regular} if, and only if, it possesses a left-inverse, i.e. $\exists$ $\bar{S}\in \MatGr{\OpFr}n$  such that $\bar{S}M= \begin{pmatrix} \ID[m] \\ \ZERO[(n-m),m] \end{pmatrix}$.
\end{itemize}
\end{prop}
\proof We only prove (i). The proof of (ii) follows the same lines and is left to the reader.
Let $M \in \Mat{\Op}nm$ be given with $n < m$. $M$ is hyper-regular if, and only if, there are matrices $S \in \MatGr{\OpFr}n$ and $T \in
\MatGr{\OpFr}m$ such that $SMT = (\ID[n], \ZERO[n,(m-n)])$. Thus, using the identity
\begin{equation*}
 \begin{pmatrix} \ID[n],& \ZERO \end{pmatrix}
  = S^{-1} \begin{pmatrix} \ID[n],& \ZERO \end{pmatrix}
  \begin{pmatrix} S & \ZERO \\ \ZERO & \ID[m-n] \end{pmatrix}
\end{equation*}
we get
\begin{equation*}
 \begin{pmatrix} \ID[n],& \ZERO \end{pmatrix} = S^{-1} (SMT) 
  \begin{pmatrix} S & \ZERO \\ \ZERO & \ID[m-n] \end{pmatrix}
  = M 
  \Bigl( T \begin{pmatrix} S & \ZERO \\ \ZERO & \ID[m-n] \end{pmatrix}\Bigr).
\end{equation*}
which proves that $M$ is hyper-regular if, and only if, it
has $\bar{T} = \Bigl( T \begin{pmatrix} S & \ZERO \\ \ZERO & \ID[m-n] \end{pmatrix}\Bigr)$ as right-inverse. \qed
%\end{proof}

We will apply the technique of row-reduction for the computation of
left- or right-inverses. A matrix $M \in \Mat{\OpFr}nm$ without zero
rows is called \emph{row-reduced} (or row-proper) if for all row
vectors $v \in \RV{\OpFr}m$ the so-called \emph{predictable degree
  property}
\begin{equation*}
  \label{eq:PDP}
  \pdeg vM = \max \{ 
  \pdeg v_j + \pdeg \row{j}M
  \mid
  j=1,\ldots,n \}
\end{equation*}
holds.  Note, that this definition differs from the usual one as given
in~\cite[Sect.~2]{Labahn_06} or~\cite[Sect.~2.2]{Zerz_07}, but is shown to
be equivalent in \cite[Lem.~A.1\,(a)]{Labahn_06} where one can easily
check that the proof also works for division rings instead of
fields. Similarly, the algorithm outlined in the proof
of~\cite[Thm.~2.2]{Labahn_06} can be easily transferred to division rings,
and we obtain that for every matrix $M \in \Mat{\Op}nm$ there exists a
matrix $S \in \MatGr{\OpFr}n$ such that the non-zero rows of $SM$ form
a row-reduced submatrix. Using the results in \cite{Middeke_11} row-reduction is of low polynomial complexity in the size of $M$ and its degree.
Directly from
the definition we obtain that the rows of a row-reduced matrix must be
linearly independent.

Row-reducedness is connected to the Popov normal form (see, \eg,
\cite[Def.~2.2]{Cheng_08}) which is essentially a row-reduced matrix
with additional properties to make it unique. One may prove that for
each matrix in $\Mat{\Op}nm$ there exists exactly one matrix Popov
form having the same row space. Also, row-reduction may be regarded as
a special case of \GB\ computation---see~\cite{Middeke_11}.

\smallskip

\color{new} 
We would like to derive a characterization of hyper-regular matrices
using row- and column-reduction.
\normalcolor
We consider now the case $n \geq m$. Let $M \in \Mat{\Op}nm$, and let
$\tilde S \in \MatGr{\Op}n$ be such that
\begin{equation*}
  \tilde SM = \begin{pmatrix} \tilde M \\ \ZERO \end{pmatrix}
\end{equation*}
where $\tilde M \in \Mat{\OpFr}km$ is row-reduced and $k$ is the
(left) row-rank of $M$ by~\cite[Thm.~A.2]{Labahn_06}. Assume first that
$M$ is hyper-regular. As discussed above, this means that all unit
vectors of $\RV{\Op}m$ are in the row-space of $\tilde{M}$. Thus, we
must have $k = m$ and by \cite[Lem.~A.1\,(c)]{Labahn_06} all rows of
$\tilde{M}$ have $\partial$-degree $0$ or below. Thus, $\tilde{M} \in
\Mat{K(\delta)}mm$. Since all rows of $\tilde{M}$ are linearly
independent, we conclude that $\tilde{M}$ has maximal (left) row rank.

Conversely, if row-reduction of $M$ yields a matrix of degree $0$ and
(left) row-rank $m$, then clearly $M$ is hyper-regular.

\medskip

There is also the analogue concept of \emph{column-reduction}. Each
matrix may be brought into column-reduced form (up to zero rows) by
right multiplication with a unimodular matrix. All the results cited
above hold with the appropriate changes. In total, we have proved the
following lemma.

\begin{lem}[\cite{Antritter_11}]\label{lem:hyperr}
  A matrix $M \in \Mat{\Op}nm$ is hyper-regular if and only if
  \begin{enumerate}
  \item $n \geq m$ and row-reduction yields a matrix of
    $\partial$-degree $0$ and left row-rank~$m$
  \item or, $n \leq m$ and column-reduction yields a matrix of
    $\partial$-degree $0$ and right column-rank~$n$.
  \end{enumerate}
\end{lem}

\section{Systems}\label{sys:sec}

Consider system (\ref{sys_lin_2}) of the introduction
with state, or partial state, $x$ of dimension $n$, input $u$ of dimension $m$ and the matrices $A\in\Mat{\Rcal}nn$ and $B\in\Mat{\Rcal}nm$, where $\Rcal$ is a ring.

For the analysis of such systems, two important objects are generally considered (see e.g. \cite{Fliess_90,Oberst_90,Willems_91,Fliess_92,Mounier_95,Polderman_98,Pommaret_99,Chyzak_05,Zerz_07}):

\begin{itemize}
\item its \emph{behavior} $\Beh \triangleq \ker (A,-B)$,
where the kernel is taken with respect to the chosen signal space\footnote{this space is not uniquely defined and may be chosen according to the specific application one is interested in, such as motion planning, tracking, etc. See section~\ref{op-sig:sec}}, where the components of the variables $x$ and $u$ are supposed to live. Here, we choose:
\begin{multline}\label{eq:sigsp}
\mathcal{S}=  \{ f \colon \mathbb{R} \to \mathbb{R}
  \mid
  \exists E_f \subseteq \mathbb{R} \text{ discrete } \colon
  f \in \mathcal{C}^\infty(\mathbb{R} \setminus E_f, \mathbb{R})
  \\
  \text{ and }
  \exists t_0 \in \mathbb{R}\;
  \forall t \in \mathbb{R}\colon t < t_0 \implies  f(t) = 0 
  \}
\end{multline}
which is an adaptation of a space introduced in \cite[Sec. 3]{Zerz_06}. Indeed, $\Sig$ is an $\Op$-module.
This fact is needed to justify that all computations done with matrices over $\OpFr$ in Section~\ref{sys:sec} actually transfer to the actions of the operators on the signals.
Note that $\Sig$ contains splines and step functions. For a discussion of this choice in regard to the motion planning problem, the reader is invited to refer to Section~\ref{ex:sec};

\item and its \emph{system module} $\MM \triangleq \Mod[\Rcal]{(A,-B)}n{(n+m)}$,
where $\RV{\Rcal}p$ is the set of row vectors of length $p$, for every $p\in \NN$, with components in $\Rcal$ and where $\RS[\Rcal]{(A,-B)}n$ is the module generated by the rows of the matrix $(A,-B)$.
\end{itemize}

Linear time-invariant differential systems, \ie \ without delay, the ring $\Rcal$ being chosen as $\RR[\partial]$ and being commutative,  are shown to be \emph{differentially flat} if, and only if, their system module $\MM$ over $\Rcal$ is free (which is equivalent to controllability \cite{Fliess_90,Fliess_92}), and a flat output is, by definition, a basis of the free module $\MM$ (see \eg\ \cite{Fliess_95,Levine_09}). Extending this approach to the context of linear differential-delay systems, Mounier, Fliess, Rudolph and others have thus proposed to replace the notion of flatness by this freeness property of the associated system module (\cite{Mounier_95,Petit_97,Rudolph_03}). Nevertheless, only few systems have a free system module.
Moreover, since the considered base ring $\Rcal=\Op = K[\delta,\partial]$ is not a principal ideal domain, freeness is in general different \eg\ from torsion freeness or projective freeness (see \eg\ \cite{Mounier_95}), with important, but otherwise unclear, practical consequences on the control possibilities of the system. However, for linear \emph{time-invariant} differential-delay systems, namely with $\Rcal= \RR [\delta,\partial]$, in virtue of the \emph{localization} property of a commutative Ore algebra, torsion freeness of $\MM$ is equivalent to the existence of a so-called \emph{liberation polynomial} $\pi \neq 0$, $\pi \in \RR [\delta]$, such that $\pi^{-1}\MM$ is free. This property is called $\pi$-freeness (see again \cite{Mounier_95}) and a basis of the free module $\pi^{-1}\MM$ is called a $\pi$-flat output. It can be interpreted as follows: if the system module, finitely generated by $m$ input variables, is torsion free, it admits a basis of the form $\pi^{-1} \{ y_{1}, \ldots, y_{m}\}$ where $y_{i}\in \Rcal$, $i=1,\ldots, m$. Consequently, every system variable can be expressed as a combination of derivatives, delays and advances of $(y_{1}, \ldots, y_{m})$. Note that the action of $\pi^{-1}$ on a variable $z$ may be interpreted as a formal power series in $\delta^{-1}$ where $\delta^{-1}z(t) = z(t+\tau)$ whose degree is bounded from above by some $N\in \NN$, and thus a finite combination of advances, or predictions, of $z$ (see Section~\ref{sec:motplan}).

\color{new}
To extend this approach to time-varying differential-delay systems without restricting the time dependence of the system coefficients to be polynomial, as in \cite{Chyzak_05,Chyzak_07} where effective Gr\"{o}bner bases techniques are used, many new difficulties appear. The main one is that localization by a single polynomial, say $p\in \Rcal$, is usually not possible  since the set made of the powers of this polynomial, namely $\{ p^{k} : k\in \NN \}$, is not in general an Ore set\footnote{We say that the set $\{ p^{k} : k\in \NN \}$ is an Ore set if, and only if, for all $k\in \NN$ and $a\in\Rcal$, there exist $n\in \NN$ and $b \in \Rcal$ such that $ap^{n}=p^{k}b$ (see, \eg,
  \cite[Sec.~5]{Cohn_00}, \cite[Sec.~2.1.6]{Mcconnell_01}).}, as shown in the following example:

\begin{ex}  
  Consider for simplicity the case $\tau=1$ and $\Rcal= K[\delta]$, \ie with $\delta a(t)\triangleq a(t-1)\delta$ for all $a\in K$. Then for $a(t) \equiv t$ and $p =
  \delta-t$ there are no exponent $n \geq 0$ and no element $b \in K$
  such that $(\delta-t)^1 b = t (\delta-t)^n$ (which is violating the
  right Ore condition for the set $\{ (\delta-t)^k \mid k \geq 0
  \}$). It is easy to prove by induction on $n$ that $\delta-t$ is not
  a left divisor of $t (\delta-t)^n$: It's obvious for
  $n=0$. Moreover, since 
  \begin{equation*}
    t (\delta - t)
    = \delta (t+1) - t^2
    = \delta (t+1) - t^2 - t + t
    = \delta (t+1) - t(t+1) + t    
    = (\delta-t) (t+1) + t,
  \end{equation*}
  we obtain for $n \geq 0$ that
  \begin{math}
    t (\delta-t)^{n+1} =  (\delta-t) (t+1) (\delta-t)^n + t (\delta-t)^n.
  \end{math}
  Thus, if $\delta-t$ was a left divisor of $t (\delta-t)^{n+1}$ it
  would also be a left divisor of $t (\delta-t)^n$. This contradicts
  the induction hypothesis.
\end{ex}

Therefore, we are lead to propose a different orientation to define $\pi$-flatness, with the two following requirements: (1) trajectory planning may be achieved in an elementary way and (2) sufficiently many linear time-varying systems satisfy this property, indeed including the class of $\pi$-free linear time-invariant systems. We thus propose a new definition based on an extension of the differential flatness one, in the spirit of \cite{Petit_97,Petit_00}, saying that all the system variables are expressible in terms of a flat output, its derivatives, delays and advances in finite number, a definition that we should call ``practical''\footnote{The term \emph{practical} does not mean at all that the definition of \emph{practical flatness} is not algebraically precise!} with regard to the motion planning application.
Moreover we introduce the definition of $\pi$-$k$-flatness, which is an analogue, in our differential-delay context, to the notion of $k$-flatness (see \cite{Martin_97,Pomet_97,Rathinam_98,Pereira_2000}).
This viewpoint is developed in the next two subsections.
\normalcolor

\subsection{Recalls on the framework}

From now on, we take System \eqref{sys_lin_2} with $\Rcal=\Op=K[\partial,\delta]$ as in Section~\ref{sec:Ore}.
Recall that $\delta$, the delay operator, is defined by $\delta f(t) = f(t-\tau)$ for all $t\in \RR$, where $\tau$ is a given positive real number and $f\in K$, and that $\partial \triangleq \frac{d}{dt}$ is the ordinary time derivative operator. The ground field  $K$ is the field of meromorphic functions over the real line and the notation $\OpFr = K(\delta)[\partial]$ is defined in Section~\ref{sec:Ore}. Furthermore, we will consider only the signal space $\Sig$ defined by (\ref{eq:sigsp}).

 We assume
that the matrix $(A,-B)$ has full (left) row rank.

In this case, System \eqref{sys_lin_2} is a differential-delay system with time-varying coefficients, the coefficients being meromorphic functions of time.

Example \ref{intro:ex} from the introduction illustrates this system class. As announced, in this example, the matrices $A,B$ are matrices over the ring $\Op=K[\delta,\partial]$, $x\in \Sig^2$ and $u\in \Sig$.

\subsection{$\pi$-flatness}\label{subsec:pi-flat}

One usual way to consider  System (\ref{sys_lin_2}) is to bring all the variables in one side, i.e. to consider the system 
$$(A,-B) \begin{pmatrix}x\\u\end{pmatrix}=\ZERO$$ 
or, in other words, 
\begin{equation}\label{sysF}
F\xi =0
\end{equation} 
with $F\triangleq (A,-B)$ and
$\xi\triangleq  \begin{pmatrix}x\\u\end{pmatrix}$.

\begin{defn}[$\pi$-flatness]\label{def:piflatness}
  The system \eqref{sysF} is called \emph{$\pi$-flat} if there
  exist $\pi \in K[\delta]$ and matrices $P \in \Op^{m\times (n+m)}$ and $Q
  \in \Op^{(n+m)\times m}$ such that
  \begin{equation}\label{defmat:eq}
    \pi^{-1} Q \Sig^{m} = \Beh = \ker F
    \qquad\text{and}\qquad
    \pi^{-1} P \; \pi^{-1} Q = \ID[m].
  \end{equation}
\end{defn}

Equivalently, there exist a polynomial $\pi \in K[\delta]$ and matrices $P$ and $Q$ of suitable dimensions over the ring $\Op$, such that $\xi= \pi^{-1}Qy$, with $y= \pi^{-1}P\xi$ for all $\xi \in \Beh$, where $y\in \Sig^{m}$ is a $\pi$-flat output. In other words, $\pi$-flatness means that all the system variables can be expressed as linear combinations of $y$ and a finite number of its delays, advances and derivatives (see Section~\ref{sec:motplan}), in an invertible way, namely the matrix $\pi^{-1}P$ admits the matrix $\pi^{-1}Q$ as right-inverse.

If $\pi \in K$, then the system is simply called \emph{flat} and $y$ a flat output.

This definition is thus an extension of the one proposed by \cite{Mounier_95,Fliess_95,Fliess_99} for finite-dimensional nonlinear ordinary differential systems or by \cite{Petit_97}. Other definitions, directly stated in terms of freeness, for time-varying linear differential-delay systems with polynomial time dependence are proposed by Chyzak, Quadrat and Robertz \cite{Chyzak_05}. It can be easily seen that freeness in the sense of Chyzak, Quadrat and Robertz implies $\pi$-flatness with $\pi\equiv 1$ and that, in the context of linear time-invariant systems, we recover the definition of $\pi$-freeness introduced by \cite{Mounier_95}.

\begin{rem}\label{rem:piobar}
Note that Definition \ref{def:piflatness} is equivalent to the existence of matrices $\bar{P}\in\Mat{\OpFr}m{(n+m)}$ and $\bar{Q}\in\Mat{\OpFr}{(n+m)}m$ such that $\bar{Q}\Sig^m=\mathcal{B}$ and $\bar{P}\bar{Q}=\ID[m]$. We recover $\pi$ by computation of the left common denominator of $\bar{P}$ and $\bar{Q}$ \cite[Prop. 5.3]{Cohn_00}. Therefore all computations can be done over $\OpFr$.
\end{rem}

We have the following proposition:

\begin{prop}\label{outputtrans:prop}
Assume that $y$ is a $\pi$-flat output of  system (\ref{sysF}). Let us set $y=\bar{T}z$ with $\bar{T} \in \MatGr{\OpFr}m$. There exists a polynomial $\kappa \in K[\delta]$ such that $z$ is a $\kappa$-flat output of  system (\ref{sysF}).
\end{prop}
\proof
By Remark \ref{rem:piobar}, it is sufficient to prove that $\bar{T}^{-1}\bar{P}$ and $\bar{Q}\bar{T}$ are related to the transformed flat output $z$. Obviously, we have $\bar{T}^{-1}\bar{P}\,\bar{Q}\bar{T} = \ID[m]$ and moreover
$
 \bar{Q}\bar{T} \CV{\Sig}m = \bar{Q} \CV{\Sig}m = \Beh.
$
With $\kappa$ being a common denominator of $\bar{T}^{-1}\bar{P}$ and $\bar{Q}\bar{T}$, the claim follows.\qed
%\end{proof}

\begin{rem}
To every $\pi$-flat system there obviously corresponds a polynomial $\pi_{0}\in K[\delta]$ of minimal degree such that the system is $\pi_{0}$-flat. 
\end{rem}

To characterize $\pi$-flat systems, we introduce the following definition:
\begin{defn}
We call $\overline{\MM}\triangleq \Mod[\OpFr]Fn{(n+m)} $ the extended system module.
\end{defn}

\begin{thm}\label{thm:piflat_smod}
We have the following equivalences:
\begin{enumerate}[(i)]
\item\label{thm:piflat_smod.i} System \eqref{sysF} is $\pi$-flat;
\item\label{thm:piflat_smod.ii} The extended system module
  $\overline{\MM}$ is free;
\item\label{thm:piflat_smod.iii} The matrix $F$ is
  hyper-regular over~$\OpFr$.
\end{enumerate}

\end{thm}
\color{new}
\begin{rem}
Note that in the linear time-invariant case Definition \ref{def:piflatness} is equivalent to the definitions of \cite{Mounier_95,Chyzak_05}. Moreover (ii) means that System  \eqref{sysF} is $\pi$-flat if, and only if, it is $\OpFr$-torsion free controllable in the terminology of \cite{Mounier_95}, i.e.  $\overline{\MM}$ does not contain $\OpFr$-torsion element.\\
Note that Theorem~\ref{thm:piflat_smod} gives a precise meaning of the underlying controllability property:  every system trajectory, defined on a finite time horizon, in $\mathcal{S}^{n+m}$, is freely parameterized by a $\pi$-flat output trajectory in $C^{\infty}([t_{0},+\infty[)$, a finite number of its derivatives, and advances and delays in finite number, according to (5), Proposition~\ref{outputtrans:prop}, and the discussion on Laurent series of Section~\ref{sec:motplan}.
\end{rem}
\normalcolor
\begin{rem}
Note that this $\pi$-flatness property is independent of the choice of signal space $\Sig$. 
\end{rem}

For the proof we need the following lemma:
\begin{lem}\label{lem:mod}
  Let $\Rcal$ be an arbitrary ring and $M \in \Mat{\Rcal}pq$, and
  let $S \in \MatGr{\Rcal}p$ and $T \in \MatGr{\Rcal}q$ be
  unimodular. Then 
  \begin{math}
    \Mod[\Rcal]{M}pq \cong \Mod[\Rcal]{(SMT)}pq
  \end{math}
  as left $\Rcal$-modules.
\end{lem}
\proof
  The map $\alpha = v \mapsto vT$ is an automorphism of
  $\RV{\Rcal}q$ because $T$ is unimodular. The kernel of the
  composition $\beta\circ\alpha$ of the projection $\beta\colon
  \RV{\Rcal}q \to \Mod[\Rcal]{SMT}pq$ with $\alpha$ is
  $\alpha^{-1}(\RS[\Rcal]{SMT}p) = \RS[\Rcal]{SM}p =
  \RS[\Rcal]{M}p$ where the second identity follows from the
  unimodularity of $S$. Since $\beta\circ\alpha$ is also surjective, by
  the first isomorphism theorem for modules (see, \eg,
  \cite[Theorem~1.17]{Cohn_00}) there is an isomorphism
  \begin{math}
    \Mod[\Rcal]{M}pq \to \Mod[\Rcal]{(SMT)}pq.
  \end{math}\qed
%\end{proof}

\proof[of Theorem~\ref{thm:piflat_smod}]
  We prove first that \eqref{thm:piflat_smod.i} is equivalent to
  \eqref{thm:piflat_smod.iii}. If the system $F\xi = 0$ is $\pi$-flat,
  there are matrices $\bar{P} \in \Mat{\OpFr}m{(n+m)}$ and $\bar{Q}
  \in \Mat{\OpFr}{(n+m)}m$ with common denominator $\pi \in
  K[\delta]\setminus\{0\}$ such that $\bar{Q} \CV{\Sig}m = \Beh = \ker
  F$ and $\bar{P} \, \bar{Q} = \ID[m]$. That means that $\bar{Q}$ is
  injective.
  \color{new}
  In mathematical terminology, injectivity of
  $\bar{Q}$, surjectivity of $F$ (onto $\im F$ in our case) and $\im
  \bar{Q} = \ker F$ are usually abbreviately expressed by saying
  that the following diagram is \emph{exact} (see, \eg,
  \cite[Page~120]{Lang_02}):
  \normalcolor
  \begin{center}
    \begin{tikzpicture}[black]
      \matrix (A) [matrix of math nodes, column sep=3.9em] {
        0 &[-2em] \Sig^m & \Sig^{n+m} & \im F &[-2em] 0. \\
      };
      \path[->, font=\scriptsize] 
      (A-1-1) edge (A-1-2)
      (A-1-4) edge (A-1-5)
      (A-1-2) edge node[above] {$\bar{Q}$} (A-1-3)
      (A-1-3) edge node[above] {$F$} (A-1-4);
      \path[transform canvas={yshift=-0.6ex},
      ->,bend left=19,font=\scriptsize] 
      (A-1-3) edge node[below] {$\bar{P}$} (A-1-2)
      (A-1-4) edge[dashed] node[below] {$\bar{E}$} (A-1-3);
    \end{tikzpicture}
 \end{center}

  Since $\bar{P} \, \bar{Q} = \ID$, by the splitting lemma (see,
  \eg,~\cite[Prop.~3.2]{Lang_02}) one can show
  that there is a matrix $\bar{E} \in \Mat{\OpFr}k{n}$ such that $F
  \bar{E} = \ID[n]$ where $n$ is the (left) rank of $\im F$ because
  the rows of $F$ are linearly independent by assumption. Therefore
  $F$ is hyper-regular. 

  Conversely, let $F$ be hyper-regular. Using Lemma~\ref{lem:hyperr}
  we may thus compute $\bar{W} \in \MatGr{\OpFr}{n+m}$ such that $F
  \bar{W} = (\ID[n], \ZERO)$. If we let $\bar{Q}$ be the last $m$
  columns of $\bar{W}$ and $\bar{P}$ the last $m$ rows of
  $\bar{W}^{-1}$, then we have $F \bar{Q} = \ZERO$ and $\bar{P}\bar{Q}
  = \ID[m]$. Extracting a common denominator $\pi$ of $\bar{P}$ and
  $\bar{Q}$, we have proved that $F \xi = 0$ is $\pi$-flat.

  \medskip
  
  Next we show that \eqref{thm:piflat_smod.iii} is equivalent to
  \eqref{thm:piflat_smod.ii}. Assume first that $F$ is
  hyper-regular. Since the unit vectors are in the column space of
  $F$, by column-reduction we obtain an invertible matrix $\bar{T} \in
  \MatGr{\OpFr}{n+m}$ such that $F \bar{T} = (\ID[n], \ZERO)$. By
  Lemma~\ref{lem:mod}, this means that
  \begin{equation*}
  \overline{\MM} =  \Mod[\OpFr]{F}{n}{(n+m)}
    \cong  \Mod[\OpFr]{F \bar{T}}{n}{(n+m)}
    \cong  \Mod[\OpFr]{(\ID[n], \ZERO)}{n}{(n+m)}
    \cong  \RV{\OpFr}m
  \end{equation*}
  which is free.

  Conversely, assume that $\overline{\MM}$ is
  free. Using, \eg, the method in~\cite[Thm.~8.1.1]{Cohn_85} we may
  obtain unimodular matrices $\bar{S} \in \MatGr{\OpFr}{n}$ and
  $\bar{T} \in \MatGr{\OpFr}{n+m}$ such that $\bar{S}F \bar{T} =
  (\bar{\Delta}, \ZERO)$ where $\bar{\Delta} = \diag{a_1,\ldots,a_{n}}
  \in \Mat{\OpFr}{n}{n}$ is a diagonal matrix. By Lemma~\ref{lem:mod},
  \begin{equation*}    
   \overline{\MM} = \Mod[\OpFr]{F}{n}{(n+m)} 
    \cong  \Mod[\OpFr]{\bar{S} F \bar{T}}{n}{(n+m)} 
    \cong  \Mod[\OpFr]{(\bar{\Delta}, \ZERO)}{n}{(n+m)} 
    \cong  \bigoplus_{j=1}^{n} \frac{\OpFr}{\OpFr a_j} \oplus \RV{\OpFr}m.
  \end{equation*}
  Since this module is free by assumption we conclude that all $a_j$
  must be units, \ie, we may assume \WLOG\ that $\bar{\Delta} =
  \ID[n]$. By Definition~\ref{def:hyperr}, $F$ is
  hyper-regular. \qed
%\end{proof}

\begin{alg}[Computation of a $\pi$-flat output]\label{alg:Fx=0}~
  \begin{description}
  \item[Input:] A matrix $F\in \Mat{\Op}n{(n+m)}$ representing the system \eqref{sysF}.
  \item[Output:] An Ore polynomial $\pi \in K[\delta]$ together with
    matrices $P \in \Mat{\Op}m{(n+m)}$ and $Q \in \Mat{\Op}{(n+m)}m$
    as in Definition \ref{def:piflatness}
    or \textsc{fail} if such matrices do not exist.
  \item[Procedure:]\ \vspace{-0.3cm}
    \begin{enumerate}
    \item Use column-reduction to check whether $F$ is
      hyper-regular. If not, then return \textsc{fail}.
    \item Else, let $\bar{W} \in \MatGr{\OpFr}{n+m}$ be such that
      \begin{equation*}
        F\bar{W} =
        \begin{pmatrix}
          \ID[n] & \ZERO[n,m]
        \end{pmatrix}.
      \end{equation*}
    \item Let
      \begin{equation*}
        \bar{Q} \triangleq \bar{W}
        \begin{pmatrix}
          \ZERO[n,m] \\ \ID[m]
        \end{pmatrix}
        \qqtext{and}
        \bar{P} \triangleq
        \begin{pmatrix}
          \ZERO[m,n] & \ID[m]
        \end{pmatrix}
        \bar{W}^{-1}.
      \end{equation*}
    \item Let $\pi \in K[\delta]$ be a common denominator of
      $\bar{P}$ and $\bar{Q}$.
    \item Set 
      $P \triangleq \pi\bar{P} \in \Mat{\Op}m{(n+m)}$ 
      and 
      $Q \triangleq \pi\bar{Q} \in \Mat{\Op}{(n+m)}m$.
    \item Return $\pi$, $P$ and $Q$.
    \end{enumerate}
  \end{description}
\end{alg}

\begin{rem}
In the above algorithm we exploit the fact that $\bar{W}^{-1}$ can be computed at the same time as $\bar{W}$ by inverting the elementary actions that compose $\bar{W}$.
\end{rem}

\subsection{$\pi$-0-flatness}

In contrast to the considerations of the previous subsection, it is sometimes necessary to keep the state and input variables separate: in  the theory of linear time-invariant systems, controllability is equivalent to the existence of Brunovsk\'{y}'s canonical form (see e.g. \cite{Brunovsky_70,Kailath_79}); the interpretation of some of their states as flat output (see e.g. \cite{Fliess_95}) shows that flat outputs do not need to depend on $u$, \ie \  there exist $P\in \RR^{m\times n}$ and $Q\in \RR^{n\times m}$ such that $y=Px$, $x=Qy$ and $PQ=1_m$. This property is called 0-flatness. The fact that $u$ can be expressed as a function of $y$ is, in this case, an immediate consequence of the system equation with $x=Qy$.
More generally:
 
\begin{defn}\label{def:pikflatness}
We say that a system is $\pi$-$k$-flat, with $k\geq 1$, if and only if there exists a $\pi$-flat output $y$ such that the maximal degree with respect to $\partial$ of the matrix $P\begin{pmatrix} \ZERO[n-m,m]\\\ID_{m}\end{pmatrix}$ is equal to $k-1$. 

We set $k=0$, by convention, if $P\begin{pmatrix} \ZERO[n-m,m]\\\ID_{m}\end{pmatrix}= 0$, \ie \  $y$ does not depend on $u$.
\end{defn}

Note that $\pi$-0-flatness is equivalent to the existence of $(\pi, P, Q)$ as in Definition~\ref{def:piflatness} such that $P= \begin{pmatrix} P_{1}&0_{m,m}\end{pmatrix}$ with $P_{1} \in \Op^{m\times n}$ and $\pi^{-1} P_{1}\pi^{-1}Q_{1}= \ID[m]$ where $Q_{1}\triangleq \begin{pmatrix}\ID[n]&0_{n,m}\end{pmatrix}Q$.

A linear flat system is not necessarily 0-flat, as shown by the following elementary example:
\begin{ex}\label{1flat:ex}
Consider the system
$$
\begin{pmatrix}
1&-\partial
\end{pmatrix}\begin{pmatrix}x\\u\end{pmatrix} =0
$$

or 
$$
x=\dot{u}.
$$
It can be easily seen that $y=u$ is a flat output. Hence the system is 1-flat.
\end{ex}
\begin{rem}\label{1flat:rem}
Note that elementary considerations show that 1-flatness is not preserved under the action of the group of unimodular matrices. Indeed, setting $\begin{pmatrix}x'\\u'\end{pmatrix}=\begin{pmatrix}0&1\\1&0\end{pmatrix}\begin{pmatrix}x\\u\end{pmatrix}$ in the previous example, we get the system 
$\begin{pmatrix}
-\partial&1
\end{pmatrix}\begin{pmatrix}x'\\u'\end{pmatrix} =0$, or $-\dot{x}'=u'$, which admits $y'=x'$ as flat output and is thus obviously 0-flat.
However, if we restrict the  transformation group to be $\MatGr{\OpFr}{n}\otimes\MatGr{\OpFr}{m}$, a group preserving the control variables, then 1-flatness is preserved.
\end{rem}

Example~\ref{1flat:ex} and Remark~\ref{1flat:rem} thus show that obtaining a characterization of $\pi$-0-flat systems is of interest.

\begin{lem}[Elimination]
If $B$ in \eqref{sys_lin_2} is hyper-regular, there exists $\tilde{M}$, left-inverse of $B$, such that System \eqref{sys_lin_2} can be decomposed according to
\begin{equation}\label{sys_lin_elim}
\begin{pmatrix}
\tilde{R}\\
\varphi^{-1} F
\end{pmatrix}
x
=\tilde{M}Ax=\tilde{M}Bu
=
\begin{pmatrix}
\ID[m]\\
\ZERO[(n-m),m]
\end{pmatrix}
u.
\end{equation}
where ${F} \in \Mat{\Op}{(n-m)}n$ and $\varphi\in K[\delta]$ with $\varphi\neq0$. 

\end{lem}

\begin{rem}
Let $\bar{F}\in\Mat{\OpFr}{p}{q}$. Let $\varphi\in K[\delta]$ be its common denominator. We have
$\bar{F}=\varphi^{-1}F$ with $F\in\Mat{\Op}pq$. Thus, $\ker \bar{F}=\ker F$. 
\end{rem}

\begin{thm}\label{thm:2}
If $B$ is hyper-regular, we have the following equivalences:
\begin{itemize}
\item[(i)] The control system \eqref{sys_lin_2} is $\pi$-0-flat; 
\item[(ii)] The extended system module $\Mod[\OpFr]F{(n-m)}n$ is free, with $F$ defined in \eqref{sys_lin_elim};
\item[(iii)] $F$ is hyper-regular over~$\OpFr$.
\end{itemize}
\end{thm}

\proof
  Let $F$ be hyper-regular over $\OpFr$. From representation
  (\ref{sys_lin_elim}) combined with Theorem~\ref{thm:piflat_smod},
  there exists $\bar{P}_{1}\in \Mat{\OpFr}mn$ and $\bar{Q}_{1}\in
  \Mat{\OpFr}nm$ such that $F\bar{Q}_{1}= \ZERO$ and
  $\bar{P}_{1}\bar{Q}_{1}= \ID[m]$. Thus,
  $\bar{P}\triangleq\begin{pmatrix}\bar{P}_{1}, \ZERO\end{pmatrix}$ and
  $\bar{Q}\triangleq\begin{pmatrix}\bar{Q}_{1}\\\tilde{R}\bar{Q}_{1}\end{pmatrix}$
  satisfy $\bar{P}\bar{Q} = \bar{P}_1\bar{Q}_1 = \ID[m]$ and
  \begin{equation}
    \label{eq:thm2}
\tilde{M} \begin{pmatrix}
           A & -B
          \end{pmatrix}
\bar{Q} =
    \tilde{M} \left( A\bar{Q}_{1} - B\tilde{R}\bar{Q}_{1} \right) 
    = \begin{pmatrix}\tilde{R}\\ \varphi^{-1}F\end{pmatrix} \bar{Q}_{1}
    - \begin{pmatrix}\ID[m]\\0\end{pmatrix}\tilde{R}\bar{Q}_{1} 
    = \ZERO.
  \end{equation}
  Equation~\eqref{eq:thm2} implies that $\bar{Q}\CV{\Sig}m \subseteq \Beh$.
  To prove the other inclusion, let $\gamma = \left(\begin{smallmatrix}\gamma_1 \\ 
  \gamma_2 \end{smallmatrix}\right) \in \Beh$. Then analogously
  \begin{equation*}
   0 = \tilde{M} \begin{pmatrix} A & -B \end{pmatrix} \gamma
   = \begin{pmatrix} \tilde{R} \\ \varphi^{-1}F \end{pmatrix} \gamma_1
   - \begin{pmatrix} \ID[m] \\ \ZERO \end{pmatrix} \gamma_2.
  \end{equation*}
Since $\ker\varphi^{-1}F=\ker F$, we obtain $\gamma_1 \in \ker F$ and $\tilde{R}\gamma_1 = \gamma_2$. Consequently, there exists $\zeta \in \CV{\Sig}m$
  such that $\gamma_1 = \bar{Q}_1 \zeta$ and thus also $\gamma_2 = \tilde{R}\bar{Q}_1\zeta$. Hence, $\gamma \in \bar{Q}\CV{\Sig}m$ 
and the other inclusion also holds.
Thus, extracting a common denominator
  $\pi$ such that $\bar{P} = \pi^{-1} P$ and $\bar{Q} = \pi^{-1}Q$
  where $P \in \Mat{\Op}mn$ and $Q \in \Mat{\Op}nm$, we see that
  $\pi$, $P$ and $Q$ fulfill Definition~\ref{def:pikflatness}. Thus,
  (iii) implies (i).

  Conversely, if (i) holds with $\pi$, $P$ and $Q$ being as in
  Definition~\ref{def:pikflatness}, then with $Q_1 = (\ID[n],
  \ZERO)Q$ and $P_1 = P \left(\begin{smallmatrix} \ID[n] \\
      \ZERO
    \end{smallmatrix}\right)$, by the same calculation~\eqref{eq:thm2} we
  obtain $\varphi^{-1}F \pi^{-1}Q_1 = \ZERO$ and $\pi^{-1}P_1\pi^{-1}Q_1 =
  \ID[m]$. Thus, $Fx = 0$ is $\pi$-flat. Thus, by
  Theorem~\ref{thm:piflat_smod} we have (iii).
  \qed
%\end{proof}

If the matrix $B$ of system \eqref{sys_lin_2} is hyper-regular, we can directly check for the existence of a $\pi$-0-flat output, as shown in Theorem \ref{thm:2}. A suitable algorithm is as follows:

\begin{alg}[Computation of a $\pi$-0-flat output]\label{alg:fo2}~
  \begin{description}
  \item[Input:] Matrices $A \in \Mat{\Op}nn$ and $B \in \Mat{\Op}nm$ with $B$ hyper-regular, representing (\ref{sys_lin_2}).

  \item[Output:] If system (\ref{sys_lin_2}) is $\pi$-0-flat, the polynomial $\pi\in K[\delta]$ and the triple $(P,Q,R)$ of the
    matrices $P \in \Mat{\Op}mn$, $Q \in \Mat{\Op}nm$ and $R\in \Mat{\Op}mm$ from Definition \ref{def:piflatness}, or \textsc{fail} if such matrices do not exist.

  \item[Procedure:]\ \vspace{-0.3cm}
  \begin{enumerate}
  \item Compute $\tilde{M} \in \MatGr{\OpFr}n$,
    s.t. $\tilde{M}B=\begin{pmatrix}\ID[m]\\\ZERO[(n-m),m]\end{pmatrix}$.
  \item Write
    \begin{equation*}
      \tilde{M} A =
      \begin{pmatrix}
        \tilde{R} \\ \varphi^{-1}F
      \end{pmatrix}
    \end{equation*}
    where $\tilde{R} \in \Mat{\OpFr}{m}n$ and $F \in \Mat{\Op}{(n-m)}n$.
   \item If Algorithm~\ref{alg:Fx=0} applied to $F$ returns $\pi_1$,
     $P_1$ and $Q_1$, then
     \begin{enumerate}
     \item Set $\bar{P} = (\pi_1^{-1} P_1, \ZERO)$ and
       \begin{equation*}
         \bar{Q} =
         \begin{pmatrix}
           \pi_1^{-1} Q_1 \\ \tilde{R} \pi_1^{-1} Q_1
         \end{pmatrix}.
       \end{equation*}
     \item Let $\pi$ be a common denominator of $\bar{P}$ and
       $\bar{Q}$ and set
       $P = \pi \bar{P} \in \Mat{\Op}mn$
       and
       $Q = \pi \bar{Q} \in \Mat{\Op}nm$.
     \item Return $\pi$, $P$ and $Q$.
     \end{enumerate}
   \item Else, return \textsc{fail}
  \end{enumerate}
  \end{description}
\end{alg}

If the matrix $B$ of system \eqref{sys_lin_2} is not hyper-regular, we can check for $\pi$-flatness by applying Algorithm~\ref{alg:Fx=0}.

\begin{rem}
Note that Algorithms \ref{alg:Fx=0} and \ref{alg:fo2} do not necessarily yield a polynomial $\pi$ of minimal degree in $\delta$. Especially, it is not guaranteed, that we obtain $\pi\in K$ if the system is flat. 
\end{rem}

%
%%%%%%%%%%%%55
%
\section{Example}\label{ex:sec}

To illustrate the results of the preceding sections and the usefulness of the concept for the feedforward controller design for linear time-delay differential systems, we demonstrate all steps for the system of Example \ref{intro:ex}.
Note that all the necessary computations can be done with a package for the computer algebra system \emph{Maple}, which has been presented in \cite{Antritter_11}. The package can be obtained from the first author upon request.

\subsection{Computation of $\pi$-flat output}\label{sec:flatout}

The matrix $B$ is hyper-regular and following Algorithm \ref{alg:fo2} we compute $\tilde{M}$ such that $\tilde{M}B=\begin{pmatrix}1\\0\end{pmatrix}$. We get
\begin{equation*}
\tilde{M}
=
\begin{pmatrix}
0&\delta^{-1}\\
1&0
\end{pmatrix}.
\end{equation*}
Thus:
\begin{equation*}
\begin{pmatrix}
\tilde{R}
\\
F
\end{pmatrix}
=
\tilde{M}A
=
\begin{pmatrix}
0&\delta^{-1}\partial
\\
\partial&a(\delta^2-\delta)
\end{pmatrix}
\end{equation*}
We now follow Algorithm~\ref{alg:Fx=0} and compute $\bar{W}$ such that $F\bar{W}=\begin{pmatrix}1&0\end{pmatrix}$ to
\begin{equation*}
\bar{W}
=
\begin{pmatrix}
0&1\\
(\delta^2-\delta)^{-1}\frac{1}{a}&-(\delta^2-\delta)^{-1}\frac{1}{a}\partial
\end{pmatrix}.
\end{equation*}
The inverse is
\begin{equation*}
\bar{W}^{-1}
=
\begin{pmatrix}
\partial&a(\delta^2-\delta)\\
1&0
\end{pmatrix}.
\end{equation*}
Subsuming the remaining computations of Algorithm \ref{alg:Fx=0} we directly obtain the following matrices in step 3 a) of Algorithm \ref{alg:fo2}, namely
\begin{equation}\label{pqac}
\bar{P}
=
\begin{pmatrix}
\begin{pmatrix}
0&1
\end{pmatrix}
\bar{W}^{-1}
&
0
\end{pmatrix}
=
\begin{pmatrix}1&0&0\end{pmatrix},
\end{equation}
and
\begin{equation}\label{eq:barQex}
\bar Q
=
\begin{pmatrix}
\ID[2] \\
\tilde R
\end{pmatrix}
\bar{W}\begin{pmatrix}0\\1\end{pmatrix}
=
\begin{pmatrix}
1\\
-(\delta^2-\delta)^{-1}\frac{1}{a}\partial 
\\
-\left(\delta^3-\delta^2\right)^{-1}\left( \frac{1}{a}\partial^2
- \frac{\dot{a}}{a^2}\partial \right)
\end{pmatrix}.
\end{equation}

We readily get $\pi=\delta^3-\delta^2$.

From \eqref{pqac}, a $\pi$-0-flat output is given by $y=x_1$. This can also be deduced from the first entry in $\bar{Q}$. From (\ref{eq:barQex}) we deduce that $x_{2}=-(\delta^2-\delta)^{-1}\left( \frac{1}{a}\dot{y}\right)$ and $u=-\left(\delta^3-\delta^2\right)^{-1}\left( \frac{1}{a}\ddot{y} - \frac{\dot{a}}{a^2}\dot{y} \right)$.

\subsection{Motion Planning}\label{sec:motplan}

In order to deal with the motion planing problem we have to clarify the action of fractions in $\delta$ on signal functions. We note that $\delta$ is an automorphism, by~\cite[Sect.~7.3]{Cohn_03}. Therefore we may
form the ring of formal Laurent series $K\(\delta\)$ in $\delta$ with
coefficients in $K$. That is, $K\(\delta\)$ consists of elements of
the form
\begin{math}
  \sum_{j \geq N} f_j \delta^j
\end{math}
where $N \in \ZZ$ and $f_j \in K$ for all $j \geq N$. Moreover, by
\cite[Prop.~7.3.7]{Cohn_03} we may embed $K(\delta)$ into
$K\(\delta\)$.

The signal space chosen in \eqref{eq:sigsp} is obviously closed under the action of any polynomial of differentiation ($\partial$) and delay ($\delta$) and it contains $\mathcal{C}^\infty$ functions with compact support in $\mathbb{R}$.

Since $\pi\in K[\delta]$, we can compute $\pi^{-1}f$, for $f\in\Sig$, by
developing $\pi^{-1}$ into a formal Laurent series $\pi^{-1} = \sum_{j
  \geq -N} a_j \delta^j$ where $N \in \mathbb{N}$ and $a_j \in
K$ for all $j \geq -N$ (see \ref{app:sec}). Then for any $f\in\Sig$ and $t\in\mathbb{R}$ we have
\begin{equation*}
  (\pi^{-1} f)(t)
  = \sum_{j \geq -N} a_j \delta^j f(t)
  = \sum_{j \geq -N} a_j f(t - j\tau).
\end{equation*}
Since $f(t - j\tau) = 0$ for $j\tau>t - t_0$ all but finitely many of the summands on the
right hand side vanish at any given $t$.
We can check that $\pi^{-1}f$ is again in $\Sig$. Thus, $\Sig$ is also closed under advances defined by inverses of $\delta$-polynomials.

We want to stress that, by the algorithm in \ref{app:sec}, the inverse of the polynomial $\pi$ as well as the above expression for $\pi^{-1} f$, contain at most $N$ predictions, where $N$ corresponds to the smallest degree of $\pi$. Moreover, as pointed out before: for all $f\in\Sig$ we have that $f(t - j\tau) = 0$ for $j\tau>t - t_0$. Therefore, at each time $t$, only a finite number of elements of the Laurent series has actually to be computed.

With these preparations, we can now generate feedforward trajectories corresponding to the following motion planning problem:
we are looking for a trajectory of $(x_{1},x_{2},u)$ starting from $x(0)=(x_1(0), x_2(0))=(0,\;0)$ and arriving at the final state $x(2)=(1,\;0)$ at time $t=2$, while $x_1$ has to be equal to zero before $t=0$ and to 1 after $t=2$ and has to be differentiable everywhere.

In order to compute explicit feedforward state and input trajectories without integrating the system differential-delay equations, we compute a reference trajectory for the flat output $y$, which is not constrained by any differential-delay equation thanks to the freeness of the extended system module, and then deduce the state and input trajectories by $\begin{pmatrix}x\\u\end{pmatrix} =\bar{Q}y$ with $\bar{Q}$ defined by (\ref{eq:barQex}). 
We take $a(t)=t+3$ and thus $\dot{a}(t)=1$. Furthermore, we take $\tau=1$, i.e. $\delta a(t)=a(t-1)$.

In order to obtain an explicit expression for $x_2=-(\delta^2-\delta)^{-1}\left( \frac{1}{a}\dot{y}\right)$, we compute the series expansion using the algorithm in \ref{app:sec} to
\begin{equation*}
\left(\delta^2-\delta\right)^{-1}
=\sum_{j\geq-1}-\delta^j.
\end{equation*}
Taking into account that $x_1$ is required to be constant for $t<0$ and $t>2$, the desired trajectory $y_d$ for $y=x_1$ will be constant at these points of time and thus $\dot{y}_d=\ddot{y}_d=0$ for $t<0$ and $t>2$. This yields
\begin{equation}\label{x2y}
x_{2,d}
=
\left(\sum_{j\geq-1}\delta^j\right)\frac{1}{a}\dot{y}_d=X_{2,d}(t)\chi_{[-1,0[}(t)+(X_{2,d}(t)+X_{2,d}(t+1))\chi_{[0,\infty[}(t+1),
\end{equation}
where $X_{2,d}(t)=\frac{1}{a(t-\lfloor t\rfloor)}\dot{y}_d(t-\lfloor t\rfloor)$, with the notation $\lfloor t \rfloor$ for the largest integer smaller than or equal to $t$, and where $\chi_{[a,b[}(t)$ is the indicator function of the interval $[a,b[$, i.e. equal to 1 on $[a,b[$ and 0 elsewhere.

We make the ansatz
\begin{equation*}
y_d(t)=(a_0+a_1t+a_2t^2+a_3t^3+a_4t^4+a_5t^5)\chi_{[0,2[}(t)+\chi_{[2,\infty[}(t)
\end{equation*}
the coefficients $a_{i}$, $i=0,\ldots, 5$ being such that $y_d(0)=x_{1}(0)=0$, $y_d(2)=x_{1}(2)=1$, $\dot{y}_d(1)= 0$ (since $x_{2}(2)=0$), $\dot{y}_d(0)=0$ and $\dot{y}_d(2)=0$ (to ensure the differentiability of $x_{1}$ at times $t=0$ and $t=2$). We readily get:
\begin{equation}\label{ydes}
y_d(t)=\left(-\frac{45}{4}t^2+\frac{35}{4}t^3-\frac{3}{4}t^5\right)\chi_{[0,2[}(t)+\chi_{[2,\infty[}(t).
\end{equation}
We now insert this expression in  \eqref{x2y} to obtain the corresponding feedforward control.
Since
\begin{equation}
\pi^{-1}=\left(\delta^3-\delta^2\right)^{-1}
=
\sum_{j\geq-2}\delta^{j}
\end{equation}
we get
\begin{equation}
u_d(t)=U_d(t)\chi_{[-2,-1[}(t)+(U_d(t)+U_d(t+1))\chi_{[-1,\infty[}(t),
\end{equation}
where $U_d(t)=-\frac{1}{a(t-\lfloor t\rfloor))}\ddot{y}_d(t-\lfloor t\rfloor)+
\frac{1}{a^2(t-\lfloor t\rfloor)}\dot{y}_d(t-\lfloor t\rfloor)$, which completes the solution of this motion planning problem.
The resulting trajectories are shown in Figure \ref{fig:ex_simres}.
Though the trajectory for $y=x_1$ is of class $C^1$, it is generated by an input $u$ which is quite irregular and complicated. We also notice that $u$ has to start $2\tau$ before the $y$-trajectory in order to satisfy the above requirements, which corresponds to the order of the Laurent series associated to $\pi^{-1}$.  
\begin{figure}[h!]
\begin{center}
\includegraphics[width=5.5cm]{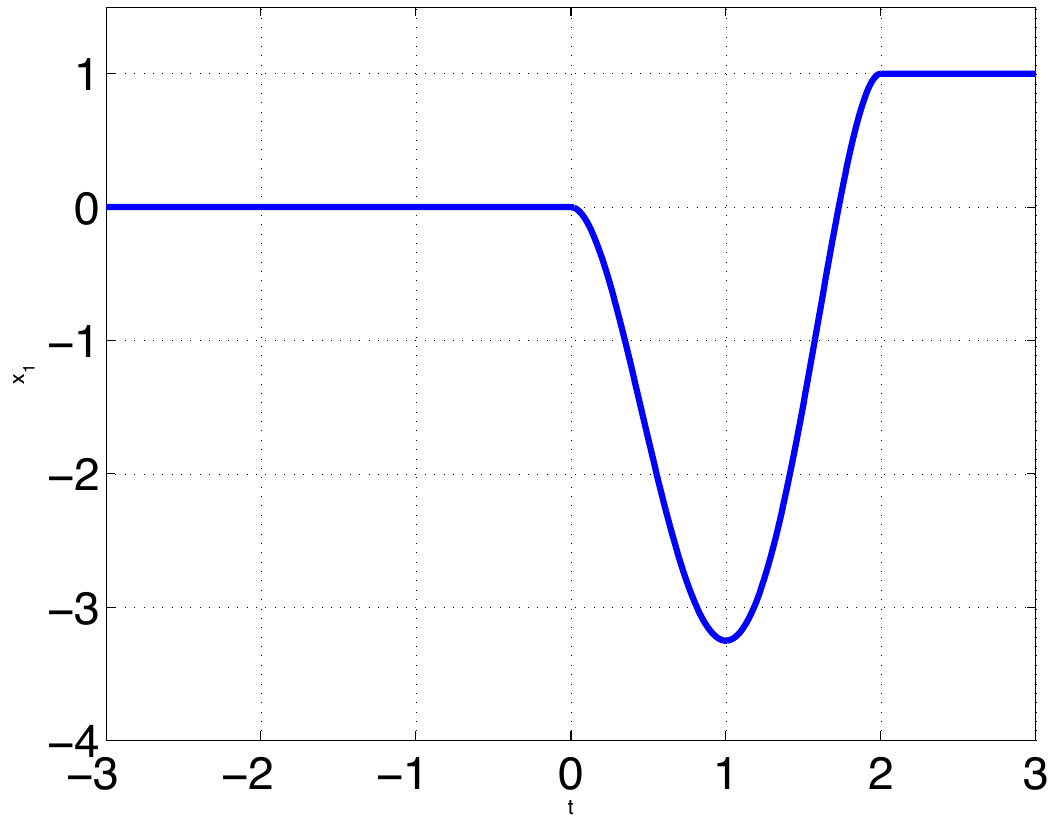}
\hspace{1cm}
\includegraphics[width=5.7cm]{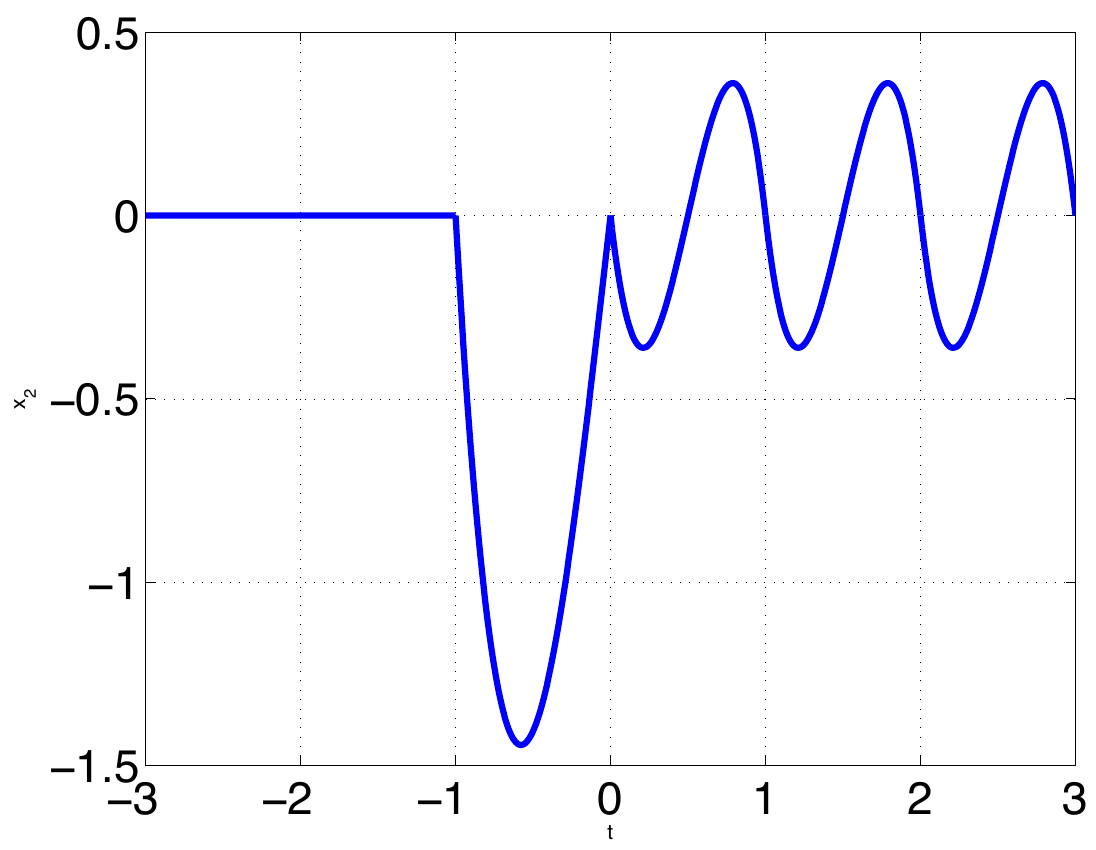}
\\
\includegraphics[width=5.6cm]{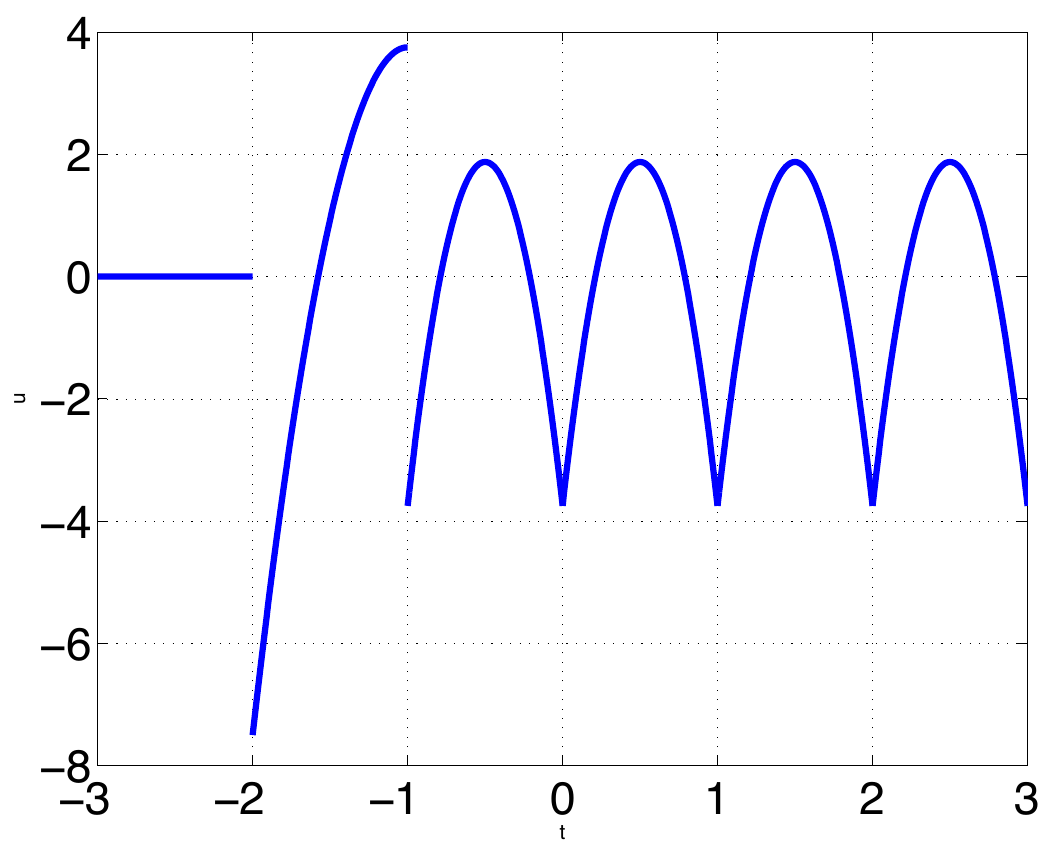}
\end{center}
\caption{Graphs for $x_1=y$, $x_2$ and $u$}
\label{fig:ex_simres}
\end{figure}

\section{Conclusion}
In this paper we have introduced a new definition of $\pi$-flatness and the important particular case of $\pi$-0-flatness. These notions have been characterized by means of the properties of matrices over Ore polynomial rings, such as hyper-regularity or row- and column-reducedness. This characterization has yielded an efficient and simple algorithm to test the existence of $\pi$- and $\pi$-0-flat outputs and compute them, via an algorithm that may be easily implemented in the computer algebra system \textit{Maple} \cite{Antritter_11}\footnote{Concerning a general approach to the implementation of meromorphic functions via Laurent-Puiseux series in Computer Algebra Systems, though not needed here, the reader may refer to \cite{Koepf_93}.}. We have illustrated the pertinence of this approach by an academic example.

Note that the description of all possible $\pi$-0-flat outputs can be achieved; it will be the subject of a forthcoming paper. This will lead to a method to determine the minimal order of the polynomial $\pi$.

Other potential directions to extend this approach may be to investigate other possible signal spaces and the consequences of our definition to feedback design.

Finally, the question of increasing the efficiency of the proposed algorithm by using quotient free computations may also be addressed.

\bibliographystyle{elsarticle-num}
%\bibliography{flatdelaybib}

\appendix
\section{Inversion of a $\delta$-Polynomial}\label{app:sec}
Let $\pi = a_n \delta^n + \ldots + a_k \delta^k \in K[\delta]$ where
$k < n$ and $a_k,\ldots,a_n \in K$. Assume that $a_k \neq 0$, i.e. that the order of $\pi$ is $k$. The
inverse $\pi^{-1}$ in $K\(\delta\)$ may be computed in the following
way: First, using the fact that
\begin{equation*}
  \pi^{-1} = \delta^{-k} (a_n \delta^{n-k} + \ldots + a_k)^{-1} 
\end{equation*}
we may assume \WLOG\ that $k = 0$. We make an ansatz $\sum_{j \geq 0}
c_j \delta^j$ for $\pi^{-1}$ and compute
\begin{equation*}
  1 
  = \pi \sum_{j \geq 0} c_j \delta^j
  =  \sum_{i=0}^n \sum_{j \geq 0} a_i \delta^i(c_j) \delta^{i+j}
  =  
  \sum_{\ell \geq 0} 
  \Bigl( \sum_{i=0}^{\min\{\ell,n\}}  a_i \delta^i(c_{\ell-i}) \Bigr)
  \delta^\ell.
\end{equation*}
Comparing coefficients and using $a_0 \neq 0$, we can compute
\begin{equation*}
  1 = a_0 c_0 \iff c_0 = a_0^{-1}
\end{equation*}
and
\begin{equation*}
  0 = \sum_{i=0}^{\min\{\ell,n\}}  a_i \delta^i(c_{\ell-i})
  \iff
  c_\ell = - a_0^{-1} \sum_{i=1}^{\min\{\ell,n\}} a_i \delta^i(c_{\ell-i})
\end{equation*}
for all $\ell \geq 1$. Note, that the left hand side depends only on
those $c_i$ which are already computed.

\end{document}